\documentclass[journal,twoside,web]{IEEEcolor}
\usepackage{generic}
\usepackage{stfloats}
\usepackage{savesym}
\usepackage{amsmath}
\usepackage{textcomp}
\savesymbol{iint}
\restoresymbol{TXF}{iint}
\usepackage{graphicx,subfigure}
\usepackage{epstopdf}
\usepackage{amssymb,amsfonts}
\usepackage{cite}
\usepackage{latexsym}
\usepackage{psfrag}
\usepackage{pifont}
\usepackage{caption}
\usepackage{color}
\usepackage{tikz}
\usepackage{bm}
\usetikzlibrary{arrows,shapes,chains}
\usepackage{bbm,mathrsfs}%
\usepackage{fancyhdr} 
\usepackage{verbatim} 
\usepackage{xcolor}

\usepackage{multirow}
\usepackage{setspace}
\usepackage{algpseudocode}
\usepackage[ruled,vlined,linesnumbered,noend]{algorithm2e}
\SetArgSty{textnormal} 

\let\labelindent\relax
\usepackage{enumitem}
\newtheorem{theorem}{Theorem}   
\newtheorem{lemma}{Lemma}
\newtheorem{problem}{Problem}

\newtheorem{example}{Example}
\newtheorem{remark}{Remark}

\newtheorem{definition}{Definition}

\newcommand{\matL}{\mathcal{L}}

\begin{document}
\title{Enforcing Opacity in Discrete Event Systems\\ via Delayed Observations}

\author{~Jiwei Wang,~Simone Baldi,~\IEEEmembership{Senior Member,~IEEE},~Wenwu Yu, ~\IEEEmembership{Senior Member,~IEEE},\\
	~and Xiang Yin,~\IEEEmembership{Member,~IEEE}\vspace{-1em}
	\thanks{This work was supported in part by the Jiangsu Provincial Scientific Research Center of Applied Mathematics grant BK20233002, and in part by the National Natural Science Foundation of China grants 62150610499, 62073074, 62073076, 62233004, 62173226. (Corresponding authors: Simone Baldi, Wenwu Yu)}
	\thanks{J. Wang is with School of Cyber Science and Engineering, Southeast University, Nanjing 210096, China (jwwang@seu.edu.cn).}
	\thanks{S. Baldi is with School of Mathematics, Southeast University, Nanjing 210096, China (simonebaldi@seu.edu.cn).}
	\thanks{W. Yu is with School of Mathematics, Southeast University, Nanjing 210096, China, and with School of Automation and Electrical Engineering, Linyi University, Linyi 276005, China (wwyu@seu.edu.cn).}
	\thanks{X. Yin is with School of Automation and Sensing, Shanghai Jiao Tong University, Shanghai 200240, China (yinxiang@sjtu.edu.cn).}
}

\maketitle

\begin{abstract}
	Artificially introducing a delay in the observations of a system can be an effective mechanism to mask the system itself, with the goal to increase its opacity and thus its security.
	This work investigates opacity in discrete event systems with delayed observations. 
	We focus on two questions: how to verify opacity under delayed	observations, and how to synthesize sensor activation policies that guarantee opacity under such delayed conditions.
	To address these questions, we first introduce the definition of opacity under delayed observation and develop a corresponding verification method.
	We then extend such analysis tool into a synthesis tool by proposing an optimization approach for designing sensor activation policies guaranteeing opacity under delayed observations.
	An example is used to illustrate the analysis and synthesis procedures.
\end{abstract}

\IEEEpeerreviewmaketitle

\section{Introduction}\label{1}

Maintaining security and privacy of a system is important in many scenarios. 
In the context of partially-observed discrete event systems (DESs), opacity has been proposed to guarantee that some secret state of the system remains hidden to inference from outside \cite{saboori2007notions}.
A system is opaque if, for every execution that reaches a secret state, there exists another execution with an identical set of observable events that does not reach a secret state: this way, it is not possible to unambiguously determine if the system is in a secret state.
Opacity of DESs has received significant attention \cite{jacob2016overview, lafortune2018history}. 
Despite the various mechanisms to enforce opacity \cite{ji2018enforcement,yin2019synthesis,mizoguchi2021abstraction,li2023opacity,duan2023edit,xie2024optimal,li2025opacity}, none of these exploits the fact that artificially introducing delayed observations (e.g., on a system that originally is not opaque) can help to make the system opaque. 
Potential applications of this extended notion of opacity span several domains \cite{jacob2016overview}. 
In these contexts, a defender may artificially introduce a delay in the sensors of the system to safeguard some secret states to an eavesdropper.

This paper aims to cover this gap: we consider opacity from the perspective of a defender who can artificially introduce a finite delay $K$ in the observations so as to mask the system to an eavesdropper.
We investigate two problems:
i) Analysis: how to verify opacity when a delay is introduced in its observations?
ii) Synthesis: how to synthesize (and possibly optimize) a sensor activation policy that guarantees opacity?
To analyze the effect of delayed observations, we extend the framework in \cite{wang2024distributed}, which is based on recording the delay of key events that help the verification of system properties \cite{wang2024fault}.
To synthesize an opaque system, we extend the dynamic mask synthesis in \cite{cassez2012synthesis}, also referred to as dynamic sensor activation, and the framework of most permissive observer in \cite{dallal2013most}. 
The observation is dynamic when the sensors can be dynamically activated or not depending on the current observation \cite{sears2016minimal,yin2019general}.

Given the above analysis and synthesis problems, this work provides the following contributions.
We address the analysis problem by introducing a definition of opacity under delayed observations, named $ K $-delayed opacity. 
Building on this definition, we propose a delay observer structure to capture opacity under dynamic and delayed observations, and we develop a corresponding verification method.
We address the synthesis problem by proposing a new container structure tailored to delayed observations.
This container structure encompasses all sensor activation policies that guarantee $ K $-delayed opacity, enabling the selection of an optimal policy according to a criterion defined in this work.
	Let us mention that the proposed $K$-delayed opacity should not be confused with $K$-step opacity \cite{saboori2007notions}, a notion that cannot address delayed observations.

The remainder of this paper is organized as follows.
Section II introduces partially-observed DESs under dynamic observation.
Section III presents the definition of $ K $-delayed opacity and the delay observer for verification.
Section IV addresses the synthesis problem.
Section V concludes the paper.

\section{Preliminaries}\label{2}

We start by recalling the basic formalism of automata.
The finite event set $E$ in DESs is modeled as an alphabet, so that a finite sequence of events in $E$ can be described as the \emph{concatenation} of a string of words in the alphabet.
A \emph{language} is a set of event strings formed from the alphabet $ E $.
Let $E^*$ denote the set of all finite strings over $E$, also known as Kleene or star closure \cite{cassandras2009introduction}.
For any string $ s\in E^* $, the length of $ s $ is denoted by $|s|$.
The empty string is denoted by $\epsilon$, with $|\epsilon| = 0$. 
The \emph{prefix-closure} of a language $L$ is defined as $\overline{L}=\{s\in E^* | \exists t\in E^*, \text{ s.t. } st\in L\}$.
A language $L$ is said to be \emph{prefix-closed} if $L=\overline{L}$.
We assume the language $L$ to be \emph{live}: $\forall s \in L, \exists e\in E $, s.t. $ se\in L$. In other words, any string in $ L $ can be extended to arbitrary length.

Automata are a fundamental framework for manipulating languages.
Consider a finite automaton defined as
\begin{equation}\label{S}
	G=(X, E, \alpha, X_0),
\end{equation}
where $X$ is the set of $ N $ finite system states, $E$ is the set of finite events, $ \alpha:X \times E^* \to 2^X $ is the transition function that describes the transition of an event string, and $X_0 \subseteq X $ is the set of initial states.
The language generated by $G$ is denoted by $\mathcal{L}(G)=\{s \in E^* | \alpha (x_0,s)!\ \exists x\in X_0 \}$, where $!$ means that the function is defined.
For any $ s\in \mathcal{L}(G) $, let us denote for brevity $\alpha(s):=\{\alpha(x,s)\mid x \in X_0\}$.
The \textit{accessible part} of $ G $ is the set of states $ \{x\in X\mid \exists s\in\matL(G),\, \text{s.t.}\, x\in\alpha(s)\} $.

A partially-observed DES is said to have \textit{dynamic observation} when the observability of the events can be controlled by a \textit{sensor activation policy} that depends on the current observation.
We divide the event set $ E $ into the observable events $ E_{o} $ and the unobservable events $ E_{uo} $.
A sensor activation policy is defined by the pair $ \Omega=(R,\Theta) $, where $ R=(Q,E,\omega,q_0) $ is a deterministic automaton that satisfies $ \matL(R)=E^* $, with $ Q $ the set of policy states and $ \omega: Q\times E^*\to Q $ the sensing transition function.
For each $ q\in Q $, $ \Theta(q)\in 2^{E_o} $ specifies the sensing decision, which determines the set of events monitored at $ q $.
To ensure that $ \Omega $ is feasible, the following feasibility condition \cite{wang2010optimal} must be satisfied:
\begin{equation}\label{FC}
\forall q,q'\in Q, e\in E: \omega(q,e)=q',[q\neq q'\Rightarrow e\in \Theta(q)],
\end{equation}
representing that sensing transitions in $ R $ occur only upon the observation of events.
This aligns with practical situations in which unobservable events cannot trigger sensing transitions.

Let us introduce a projection operator to obtain the observation of an event string:
$\forall se\in E^*,$
\begin{equation}
P_{\Omega}(\epsilon)= \epsilon,\ P_{\Omega}(se)=\left\{
\begin{aligned}
& P_{\Omega}(s)e, &\text{if} \ e\in \Theta(\omega(s)), \\
& P_{\Omega}(s),  &\text{if} \ e\notin \Theta(\omega(s)).
\end{aligned}
\right.
\end{equation}
The projection operator $ P_{\Omega} $ can also be applied to a language, that is, $ \forall L \subseteq E^* $, $ P_{\Omega}(L)=\{ s\in E_o^*| \exists s' \in L, \text{ s.t. } s=P_{\Omega}(s') \} $.
In other words, for any system trajectory $ s\in \mathcal{L}(G) $, $ \Theta(\omega(s)) $ and $ P_{\Omega}(s) $ return the current sensing decision and the observation under $ \Omega $, respectively.
Using $ P_\Omega $, the feasibility condition \eqref{FC} can also be stated as
\begin{equation}\label{FC2}
\forall s,s' \in \mathcal{L}(G),\ P_{\Omega}(s) = P_{\Omega}(s') \Rightarrow \omega(s) = \omega(s'),
\end{equation}
implying that when two system trajectories are indistinguishable, they must result in the same sensing decision.

Given $ P_{\Omega} $, we consider the operation $ \zeta^n_{\Omega}: \mathcal{L}(G) \rightarrow \mathcal{L}(G) $ to handle delayed information: $ \forall s \in \mathcal{L}(G) $,
\begin{equation}\label{key}
\zeta^n_{\Omega}(s) = \{s'' \in \overline{s} \mid \exists s' \!\in\! \overline{s}: |s'| \!\geq\! |s|-n, \text{s.t.} P_{\Omega}(s'') \!=\! P_{\Omega}(s') \},
\end{equation}
where $ n\in\mathbb{N} $ is the number of delay steps.
Intuitively, $ \{P_{\Omega}(s')\mid s' \in \overline{s}\wedge |s'| \geq |s|-n \} $ represents the set of possible observations that eavesdroppers may receive with delay after the occurrence of $ s $, and $ \zeta^n_{\Omega}(s) $ collects all prefixes of $ s $ that yield the same observation under $ \Omega $ as an element in this set.

\begin{example}\label{ESAP}
	Throughout this work, we consider the automaton $ G $ in Fig. \ref{fs} describing a simplified location-detection problem. 
	The state set $X = \{0,1,2,3\}$ corresponds to $N=4$ locations, while events $e_1$ and $e_2$ denote detection signals received from sensor $ 1 $ and sensor $ 2 $, respectively, with $E = E_o = \{e_1, e_2\}$. 
	Consider the sensor activation policy $ \Omega $ shown in Fig. \ref{fsap}, describing that sensor $ 2 $ is kept continuously active.
	Given the string $ e_1e_2 $, we have $ \Theta(0)=\{e_2\} $, $ P_{\Omega}(e_1e_2) = e_2 $, $ \zeta^0_{\Omega}(e_1e_2) = \{e_1e_2\} $ and $ \zeta^1_{\Omega}(e_1e_2) = \{\epsilon,e_1,e_1e_2\} $.
	\hfill \ensuremath{\Box}
\end{example}

\section{Opacity with Delayed Observations}\label{3}
	
In this section, we propose an opacity concept called $K$-delayed opacity, which incorporates artificially introduced delays as a mechanism for protecting the secret states.

\subsection{$ K $-delayed Opacity }\label{SDO}
		
Let us first recall the notion of opacity.
	
\begin{definition}\label{O}
	(Opacity \cite{saboori2007notions}) 
	Given a system model $G$ in \eqref{S} with a set of secret states $ X_S\subseteq X $ and a sensor activation policy $ \Omega $, the live language $ \mathcal{L}(G) $ is opaque w.r.t. $ X_S $ and $ P_{\Omega} $ if $ \forall s \in \mathcal{L}(G): \alpha(s)\subseteq X_S $,
	\begin{equation}\label{OC}
	\exists s'\in \mathcal{L}(G): P_{\Omega}(s')=P_{\Omega}(s),\text{ s.t. } \alpha(s')\not\subseteq X_S.
	\end{equation}
\end{definition}
\vspace*{1ex}

In other words, an opaque system is such that for any event string that leads to a secret state, there exists another string with the same observation leading to a non-secret state.

\begin{example}\label{E}
	Let the set of secret states of $ G $ in Fig. \ref{fs} be $ X_{S} = \{2,3\} $, representing that the defender wishes to protect the information of the current location being in states $2$ or $3$.
	Let the sensor activation policy $ \Omega $ be as in Fig. \ref{fsap}.
	Since $ \alpha(e_1e_2)=\{2\}\subseteq X_S, \alpha(e_1e_2e_1)=\{3\}\subseteq X_S $ and $ \{s'\in\matL(G)\mid P_\Omega(s') = P_\Omega(e_1e_2)\}=\{e_1e_2,e_1e_2e_1\} $, we have that $ \mathcal{L}(G) $ is not opaque w.r.t. $ X_S $ and $ P_{\Omega} $.
	Note that the transition $ (1,e_2,2) $ is key to an eavesdropper to know that the system will be in secret states until the next $ e_2 $ is observed.
	\hfill \ensuremath{\Box}
\end{example}

A system lacking opacity may become opaque when introducing an observation delay, motivating the following definition.

\begin{definition}\label{KO}
	($ K $-delayed opacity).
	Given a system model $G$ in \eqref{S} with a set of secret states $ X_S\subseteq X $ and a sensor activation policy $ \Omega $, the live language $ \mathcal{L}(G) $ is $ K $-delayed opaque w.r.t. $ X_S $ and $ P_{\Omega} $ if $ \forall s \in \mathcal{L}(G): |s|\geq K \wedge \alpha(s)\subseteq X_S $,
	\begin{equation}\label{KDOC}
	\exists s'\in \mathcal{L}(G):P_{\Omega}(s') \in P_{\Omega}(\zeta^K_{\Omega}(s)),\, \text{s.t.}\, \alpha(s')\not\subseteq X_S.
	\end{equation}
\end{definition}
\vspace{0.6ex}

If \eqref{KDOC} holds, the string $ s $ cannot be distinguished from another string $ s' $ that does not lead to secret states under delayed observations.
Thus, none of the events required to unambiguously determine if the system is in a secret state can be received timely by the eavesdropper.

As $ K $ increases, the strings that cannot be distinguished from the string $ s: \alpha(s)\subseteq X_S $ become more.
It can be verified from \eqref{KDOC} that $ K' $-delayed opacity $ \Rightarrow $ $ K $-delayed opacity when $ K \geq K' $, i.e., larger delay increases opacity. Clearly, opacity implies $K$-delayed opacity

\begin{figure}[t]
	\centering
	\subfigure[$ G $]
	{\label{fs}
		\includegraphics[height=2.5cm]{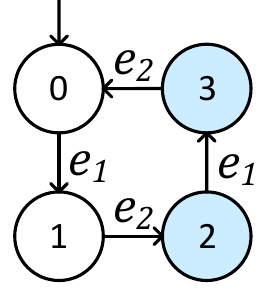} }
	\subfigure[$ \Omega $]
	{\label{fsap}
		\includegraphics[height=1.5cm]{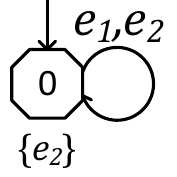}} 
	\subfigure[$ \mathcal{O}_{\Omega}^{K_1}(G) $]
	{\label{fdo1}
		\includegraphics[height=2.5cm]{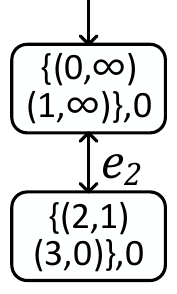}} 
	\subfigure[$ \mathcal{O}_{\Omega}^{K_2}(G) $]
	{\label{fdo2}
		\includegraphics[height=2.5cm]{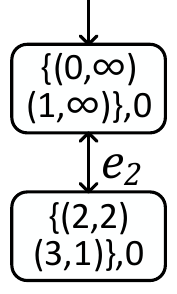}} 
	\caption{(a) System model $ G $, with secret states $ X_{S} = \{2,3\} $ marked in blue;
	(b) Sensor activation policy $\Omega$, maintaining sensor 2 continuously active;
	(c)-(d) Delay observers $ \mathcal{O}_{\Omega}^{K_1}(G) $ and $ \mathcal{O}_{\Omega}^{K_2}(G) $ with delay $ K_1\!=\!1 $ and $ K_2\!=\!2 $.}
	\label{fdo}
\end{figure}

\subsection{Delay Observer}\label{SKTCD}

Similar to opacity, verifying $K$-delayed opacity cannot be accomplished by examining all possible string of events of a system. 
It is essential to incorporate delay information in a verification method.
We incorporate delay information via a delay observer: we stress on the fact that existing delay observers based on recording the delay of at most one observable event \cite{wang2024fault} cannot be applied to the scenario under consideration, where 
multiple delays within an observation may need to be recorded.
Such difference requires to extend the delay observer structure.

\begin{algorithm}[!t]
	\caption{Construction of the delay observer $\mathcal{O}^K_{\Omega}(G)$}
	\label{AOS}
	\KwIn{$G \!=\! (X,\;\!\! E,\;\!\! \alpha,\;\!\! X_{0}), R \!=\! (Q,\;\!\! E,\;\!\! \omega,\;\!\! q_{0}),\;\!\! \Theta,\;\!\! K,\;\!\! E_o,\;\!\! X_S$;}
	\KwOut{$\mathcal{O}^K_\Omega(G) = (Y, E_o, \beta, y_{0})$;}
	
	$y_0 \gets (\emptyset, q_0)$; $u_0 \gets \infty$\;
	\If{$\phi(X_0, E_o)\subseteq X_S$}{
		$u_0 \gets K$\;
	}
	\For{$x_0 \in X_0$}{
		$I(y_0)\gets I(y_0)\cup(x_0,u_0)$;
		\texttt{REC-UNOBS}($x_0,u_0,y_0$)\;
	}
	$Y \gets \{y_0\}$;
	\texttt{REC-OBS}($y_0$)\;
	\textbf{Return} $\mathcal{O}^K_{\Omega}(G)$\;
	
	\vspace{0.6ex}
	\textbf{Procedure} \texttt{REC-UNOBS}($x, u, y$)\\
	\For{$e \in E \setminus \Theta(D(y)):\alpha(x, e)!$}{
		\For{$x' \in \alpha(x, e)$}{
			\eIf{$0 < u < \infty$}{
				$u' \gets u - 1$\;
			}{
				$u' \gets u$\;
			}
			\If{$(x', u') \notin I(y)$}{
				$I(y) \gets I(y) \cup \{(x', u')\}$\; \texttt{REC-UNOBS}($x', u', y$)\;
			}
		}
	}
	
	\vspace{0.6ex}
	\textbf{Procedure} \texttt{REC-OBS}($y$)\\
	$\iota \gets \{x \in X \mid (x, u) \in I(y)\}$;
	$E^{\mathrm{Rec}} \gets \emptyset$\;
	\For{$e \in \Theta(D(y))$}{
		$\iota^* \gets \{x \in X \mid \exists x' \in \iota, \, $s.t.$ \, x \in \alpha(x', e)\}$\;
		\If{$\phi(\iota^*, \Theta(\omega(D(y), e))) \subseteq X_S$}{
			$E^{\mathrm{Rec}} \gets E^{\mathrm{Rec}} \cup \{e\}$\;
		}
		$y^e \gets (\emptyset, \omega(D(y), e))$\;
	}
	\For{$(x, u) \in I(y)$}{
		\For{$e \in \Theta(D(y)):\alpha(x, e)!$}{
			\uIf{$u = \infty \wedge e \in E^{\mathrm{Rec}}$}{
				$u' \gets K$\;
			}
			\uElseIf{$e \notin E^{\mathrm{Rec}}$}{
				$u' \gets \infty$\;
			}
			\Else{
				$u' \gets u - 1$\;
			}
			\For{$x' \in \alpha(x, e)$}{
				\If{$(x', u') \notin I(y^e)$}{
					$I(y^e) \gets I(y^e) \cup \{(x', u')\}$\;
					\texttt{REC-UNOBS}($x', u', y^e$)\;
				}
			}
		}
	}
	\For{$e \in \Theta(D(y))$}{
		Add $\beta(y, e) = y^e$ to $\mathcal{O}_{\Omega}(G)$\;
		\If{$y^e \notin Y$}{
			$Y \gets Y \cup \{y^e\}$;
			\texttt{REC-OBS}($y^e$)\;
		}
	}
\end{algorithm}

The following deterministic observer structure is proposed
\begin{equation}\label{DO}
\mathcal{O}^K_{\Omega}(G) = (Y, E_o, \beta, y_{0}),
\end{equation}
where $ Y\in 2^{X \times (\mathbb{Z} \cup \{\infty\})}\times Q $ is the observer state space, with $ \mathbb{Z} $ the set of integers.
For notational convenience, let us decompose each $ y\in Y $ as $ y=(I(y),D(y)) $.
Namely, $ I(y)\in 2^{X \times (\mathbb{Z} \cup \{\infty\})} $ represents the state estimation set: for any $ (x, u) \in I(y) $, the system state $ x $ is assigned with a delay value $u$ counting the maximum number of steps before receiving the key events (i.e., the observable events allowing to distinguish the states in $ X_S $ from other states). 
Meanwhile, $ D(y)\in Q $ specifies the current sensing decision $ \Theta(D(y))\in2^{E_o} $.
For a state set $ \iota\subseteq X $ and a sensing decision $ \theta\subseteq E_o $, denote $
\phi(\iota, \theta)=\{x\in X \mid \exists x'\in \iota, s\in (E\setminus \theta)^*, \text{s.t.} \, \alpha(x',s)=x\},
$ representing the set of states that can be reached unobservably from a state in $ \iota $ under the sensing decision $ \theta $.
The transition function $ \beta: Y \times E_o \to Y $ and the initial state $ y_{0} $ in \eqref{DO} are constructed as in Algorithm \ref{AOS}.
Two recording procedures are proposed to handle the transitions involving unobservable and observable events, respectively.
Specifically, \texttt{REC-UNOBS}($x, u, y$) appends a delay value to each system state, storing the result in $ I(y) $. 
Meanwhile, \texttt{REC-OBS}($y$) identifies new transitions and observer states, where the delay values are appended to the next initial system states.

Note that for any $ (x,u) \in I(y) $, there always exists $ s\in\mathcal{L}(G): x\in\alpha(s) $ such that $ y=\beta(P_{\Omega}(s)) $.
If there exists another state $ (x',u') \in I(y) $, then there exists $ s'\in\mathcal{L}(G): x'\in\alpha(s') $ such that $ P_{\Omega}(s)=P_{\Omega}(s') $.
For each $ s\in\mathcal{L}(G): \alpha(s)\subseteq X_S $, the delay observer $\mathcal{O}^K_{\Omega}(G)$ records the delay of the events in $ E_o $ that help to distinguish $ s $ from the system trajectories $ s'\in\mathcal{L}(G): \alpha(s')\not\subseteq X_S $. 
A delay value ``$ 0 $'' indicates that the observation allowing to distinguish the states in $ X_S $ from other states has been received.
In line with existing observers \cite{cassandras2009introduction}, also the observer in Algorithm \ref{AOS} has worst-case complexity $ O(2^{2N}) $, where $ N $ is the number of system states.
\vspace*{-3ex}

\subsection{Verification of $ K $-delayed Opacity}\label{SV}

For the delay observer $ \mathcal{O}^K_{\Omega}(G) = (Y, E_o, \beta, y_{0}) $, define the verification function $ \psi : Y \rightarrow \{O,T\} $ as follows:
$\forall y \in Y, $
\vspace*{-0.6ex}
\begin{equation}\label{vf}
	\psi(y) = \left\{
	\begin{aligned}
	O , & \quad \text{if } 	\forall (x,u)\in I(y), u > 0; \\
	T , & \quad \text{otherwise},
	\end{aligned}
	\right.
\end{equation}
where $O$ stands for `opacity' and $T$ for `transparency'.
Verification of $ K $-delayed opacity is along the following result.
\begin{theorem}\label{TKDO}
	Given a system model $G$ in \eqref{S} with secret state set $X_S \subseteq X $, let $ \Omega $ be a sensor activation policy and $ \mathcal{O}_{\Omega}^K(G) $ be the delay observer built via Algorithm \ref{AOS}.
	Then, $ \mathcal{L}(G) $ is $ K $-delayed opaque w.r.t. $X_S$ and $ P_{\Omega} $ if and only if 
	\begin{equation}\label{iffdc}
	\forall y \in Y, \psi(y) = O.
	\end{equation}
\end{theorem}
\vspace{1ex}

\begin{proof}
	(\underline{Sufficiency}):
	Suppose that \eqref{iffdc} holds.
	Then, for any $ s\in \mathcal{L}(G) $ such that $ |s|\geq K $, $ \psi(\beta(P_{\Omega}(s)))=O $, that is, $ \forall (x,u)\in\beta(P_\Omega(s)), u>0 $.
	We claim that there exists $ s_\infty\in \bar{s}: |s|-|s_\infty|\leq K $ such that
	\begin{equation}\label{infty}
	\exists (x,u)\in I(\beta(P_\Omega(s_\infty))), \text{ s.t. } u=\infty.
	\end{equation}
	Suppose on the contrary that such a string $ s_\infty $ does not exist.
	Then, $ \forall t\in \bar{s}:|s|-|t|\leq K $,
	\begin{equation*}
	\forall (x,u)\in \beta(P_\Omega(t)), u\leq K.
	\end{equation*}
	According to the recording strategy in Algorithm \ref{AOS} (cf. lines 11-14 and lines 27-32), the delay value $ u<\infty $ will be updated to $ u-1 $ at each step and will become $ 0 $ within $ K $ steps, resulting in a contradiction.\\
	Note that the existence of $ s_\infty\in \bar{s} $ satisfying \eqref{infty} implies that there exists a string $ s'\in\matL(G) $ such that $ P_{\Omega}(s') = P_{\Omega}(s_\infty) \wedge \alpha(s')\not\subseteq X_S $ (cf. lines 19-23 and lines 29-30).
	By $ |s|-|s_\infty|\leq K $, we have $ s_\infty\in \zeta_{\Omega}^{K}(s) $, implying $ P_{\Omega}(s')=P_{\Omega}(s_\infty) \in P_{E_o}(\zeta^K_{\Omega}(s)) $.
	Hence, \eqref{KDOC} holds for any $ s\in\matL(G):|s|\geq K $.
	We conclude that $ \mathcal{L}(G) $ is $ K $-delayed opaque w.r.t. $X_S$ and $ P_{E_o} $.
	\\
	(\underline{Necessity}):
	Suppose that $ \mathcal{L}(G) $ is $ K $-delayed opaque w.r.t. $X_S$ and $ P_{\Omega} $.
	We will show that for any $ s\in \matL(G) $, $ x\in \alpha(s) $, the delay value of $ x $ will always be greater than $ 0 $.\\
	We first consider the set of string $ \{s \in \mathcal{L}(G) \mid |s|< K\} $.
	In Algorithm \ref{AOS}, the delay value of each initial state is $ \infty $ or $ K $.
	According to the strategy of delay record (cf. lines 11-14 and lines 27-32), the delay value of each state $ x\in\alpha(s) $ will remain greater than $ 0 $ within $ K-1 $ steps.\\
	We now show that for any $ s_\infty \in \mathcal{L}(G) $, if \eqref{OC} holds for $ s=s_\infty $, then the delay value of any $ x\in\alpha(s_\infty) $ is $ \infty $.
	Denote $ s_\infty=s_0e_0s_1 $ such that $ P_{\Omega}(s_0e_0)=P_{\Omega}(s_\infty) $ and $ e_0\in \Theta(\omega(s_0)) $.
	According to the recording strategy for state estimation set not contained in $ X_S $ (cf. lines 19-23 and lines 29-30), the delay value of any $ x\in\alpha(s_0e_0) $ is $ \infty $.
	Since $ P_{\Omega}(s_0e_0)=P_{\Omega}(s_\infty) $, we have that the delay value of any $ x\in\alpha(s_\infty) $ remains $ \infty $ (cf. lines 9-14).
	Hence, for any $ s_\infty \in \mathcal{L}(G) $ such that \eqref{OC} holds with $ s=s_\infty $, the delay value of the states in $ \alpha(s_\infty) $ is $ \infty $.\\
	Finally, since $ \mathcal{L}(G) $ is $ K $-delayed opaque w.r.t. $X_S$ and $ P_{\Omega} $, \eqref{KDOC} holds for any $ s \in \mathcal{L}(G): |s|\geq K \wedge \alpha(s)\subseteq X_S $, implying that there exists $ s_\infty\in \bar{s}: |s|-|s_\infty|\leq K $ such that \eqref{OC} holds for $ s=s_\infty $.
	Then, the delay value of any state $ x_\infty\in\alpha(s_\infty) $ is $ \infty $.
	According to the strategy of delay record (cf. lines 11-14 and lines 27-32), the delay value of the states in $ \alpha(s) $ will remain greater than $ 0 $ starting from $ (x_\infty,\infty) $ within $ K $ steps.\\
	Hence, for any $ y\in Y $ and $ (x,u)\in y $, we have $ u>0 $, implying $ \psi(y) = O $, i.e., \eqref{iffdc} holds.
\end{proof}

\begin{example}\label{EKDO}
	Let us introduce delay $ K_1=1 $ or $ K_2=2 $ in the observations of system $G$ in Fig. \ref{fs}. 
	Algorithm \ref{AOS} returns the delay observers $ \mathcal{O}_{E_o}^{K_1}(G) $ and $ \mathcal{O}_{E_o}^{K_2}(G) $ shown in Fig. \ref{fdo1} and Fig. \ref{fdo2}, where the delays of the key transition $ (1,e_2,2) $ are recorded.
	Using \eqref{vf} in Theorem \ref{TKDO}, we have $ \psi(\{(2,1),(3,0)\})=T $ for $ \mathcal{O}_\Omega^{K_1}(G) $ and $ \psi(y)=O, \forall y \in Y, $ for $ \mathcal{O}_\Omega^{K_2}(G) $.
	Hence, $ \mathcal{L}(G) $ is not $ K $-delayed opaque with $ K = 1 $, but it is $ K $-delayed opaque with $ K = 2 $ w.r.t. $ X_S $ and $ P_{\Omega} $.
	Hence, system $G$ originally lacking opacity (cf. Example \ref{E}) exhibits $ K $-delayed opacity with $K=2$.
	\hfill \ensuremath{\Box}
\end{example}

\section{Synthesis and Optimization of Sensor Activation Policy}
\label{optimize}


Instead of a given sensor activation policy, it might be desirable to design the sensor activation policy, e.g., balancing between acquiring information and preserving secrecy.
In this section, we aim to design a policy that monitors as many events as possible while guaranteeing opacity of the system. 
To formalize this objective, for two sensor activation policies $ \Omega_1=(R_1,\Theta_1),\Omega_2=(R_2,\Theta_2) $, we write $ \Omega_1 \subseteq \Omega_2 $ if $ \forall s\in E^*, \Theta_1(\omega_1(s)) \subseteq \Theta_2(\omega_2(s)) $, and $ \Omega_1 \subset \Omega_2 $ if $ \Omega_1 \neq \Omega_2 $. 
We are now in a position to formulate the problem.

\begin{problem}\label{P}
	Given a system model $G$ in \eqref{S} and a delay $K$, let $X_S\subseteq X$ be the set of secret states.
	Find a sensor activation policy $ \bar{\Omega} $ such that
	\begin{itemize}
		\item[i)] $ \mathcal{L}(G) $ is $ K $-delayed opaque w.r.t. $ X_S $ and $ P_{\bar{\Omega}} $;
		\item[ii)] there exists no other policy $ \Omega $ such that $ \bar{\Omega}\subset\Omega $ and $ \mathcal{L}(G) $ is $ K $-delayed opaque w.r.t. $ X_S $ and $ P_{\Omega} $.
	\end{itemize}
\end{problem}

To solve Problem \ref{P}, we leverage the verification function in \eqref{vf}, extending the delay observer structure to facilitate the application of Theorem \ref{TKDO} throughout the optimization process.
For the system model $ G $ in \eqref{S} and the verification function $ \psi $, let us define the following container automaton
\begin{equation}\label{container}
\mathcal{O}_{\psi}^K(G) = (Y_{\psi}, E_o, \beta_{\psi}, Y_{\psi,0}),
\end{equation}
with state space $ Y_{\psi} \in 2^{X \times (\mathbb{Z} \cup \{\infty\})}\times 2^{E_o} $.  
Similar to \eqref{DO}, let us decompose each $ \bar{y}\in Y_\psi $ as $ \bar{y}=(I(\bar{y}),D_\psi(\bar{y})) $, where $ I(\bar{y})\in 2^{X \times (\mathbb{Z} \cup \{\infty\})} $ represents the state estimation set with delay value, and $ D_\psi(\bar{y})\in 2^{E_o} $ represents the current sensing decision.
The transition function $ \beta_{\psi}: Y_{\psi}\times E_o \rightarrow 2^{Y_{\psi}} $ and the set of initial states $ Y_{\psi,0} $ in \eqref{container} are constructed in Algorithm \ref{AO}.
We stress that, instead of a sensor activation policy, the new structure is based on the verification function $ \psi $, that is, 
\begin{equation}\label{Opsi}
\forall \bar{y} \in Y_{\psi}, \psi(\bar{y}) = O,
\end{equation}
where the domain of the verification function $ \psi $ is extended to include $ Y_\psi $ with a slight abuse of notation.
Note that \eqref{Opsi} is reflected in lines 9-11 and lines 39-42 of Algorithm \ref{AO}.
While \texttt{REC-UNOBS} is as in Algorithm \ref{AOS}, a new procedure \texttt{REC-OBS$ _\psi $} is proposed to traverse the state space and record the delays until state $ \bar{y} \in Y_{\psi} $ is such that $ \psi(\bar{y}) = T $.
The final step of Algorithm \ref{AO} (cf. lines 12-13) removes states where no feasible sensing decision exists to enable further transitions.

\begin{algorithm}[!t]
	\caption{Construction of $\mathcal{O}^K_{\psi}(G)$}
	\label{AO}
	\KwIn{$G = (X, E, \alpha, X_{0}), \psi, K, E_o, X_S$}
	\KwOut{$\mathcal{O}^K_{\psi}(G) = (Y_{\psi}, E_o, \beta_{\psi}, Y_{\psi,0})$}
	
	$Y_{\psi,0} \gets \emptyset$;
	$Y_{\psi} \gets \emptyset$\;
	
	\For{$\theta \in 2^{E_o}$}{
		$\bar{y}_0^\theta \gets (\emptyset, \theta)$; $u_0 \gets \infty$\;
		\If{$\phi(X_0,\theta)\subseteq X_S$}{
			$u_0 \gets K$\;
		}
		\For{$x_0 \in X_0$}{
			$I(\bar{y}_0^\theta) \gets I(\bar{y}_0^\theta) \cup \{(x_0, u_0)\}$\;
			\texttt{REC-UNOBS}($x_0, u_0, \bar{y}_0^\theta$)\;
		}
		\If{$\psi(\bar{y}_0^\theta) = O$}{
			$Y_{\psi,0} \gets Y_{\psi,0} \cup \{\bar{y}_0^\theta\}$;
			$Y_{\psi} \gets Y_{\psi} \cup \{\bar{y}_0^\theta\}$\;
			\texttt{REC-OBS}$_\psi$($\bar{y}_0^\theta$)\;
		}
	}
	
	\While{$ \exists \bar{y}\in Y_{\psi}: (x,u)\in I(\bar{y}), e \in D_\psi(y),\, $s.t. [$ \alpha(x,e)! $ and $ \beta_{\psi}(\bar{y},e) $ is undefined]}{
		$Y_{\psi} \gets Y_{\psi} \!\setminus\! \{\bar{y}\}$;
		Take the accessible part of $\mathcal{O}^K_{\psi}(G)$\;
	}
	\textbf{Return} $\mathcal{O}^K_{\psi}(G)$\;
	
	\vspace{0.6ex}
	\textbf{Procedure} \texttt{REC-OBS}$_\psi$($\bar{y}$)\\
	$\iota \gets \{\,x \in X \mid (x,u)\in I(\bar{y})\}$\;
	\For{$\theta \in 2^{E_o}$}{
		$E^{\rm Rec}_{\theta} \gets \emptyset$\;
	}
	\For{$e \in D_{\psi}(\bar{y})$}{
		$\iota^* \gets \{\,x\in X \mid \exists x'\in \iota,\; x\in\alpha(x',e)\}$\;
		\For{$\theta \in 2^{E_o}$}{
			\If{$\phi(\iota^*, \theta)\subseteq X_S$}{
				$E^{\rm Rec}_{\theta} \gets E^{\rm Rec}_{\theta} \cup \{e\}$\;
			}
			$\bar{y}^{e,\theta} \gets (\emptyset, \theta)$\;
		}
	}
	
	\For{$(x,u)\in I(\bar{y})$}{
		\For{$ e\in D_\psi(\bar{y}):\alpha(x,e)! $}{
			\For{$\theta \in 2^{E_o}$}{
				\If{$u = \infty \wedge e \in E^{\rm Rec}_{\theta}$}{
					$u^{\theta} \gets K$\;
				}
				\ElseIf{$e \notin E^{\rm Rec}_{\theta}$}{
					$u^{\theta} \gets \infty$\;
				}
				\Else{
					$u^{\theta} \gets u - 1$\;
				}
				\For{$x' \in \alpha(x,e)$}{
					\If{$(x', u^{\theta}) \notin I(\bar{y}^{e,\theta})$}{
						$I(\bar{y}^{e,\theta}) \gets I(\bar{y}^{e,\theta}) \cup \{(x', u^{\theta})\}$\;
						\texttt{REC-UNOBS}($x', u^{\theta}, \bar{y}^{e,\theta}$)\;
					}
				}
			}
		}
	}
	\For{$e \in D_{\psi}(\bar{y})$}{
		\For{$ \theta \in 2^{E_o}: \psi(\bar{y}^{e,\theta})=O $}{
			Add $\beta_{\psi}(\bar{y}, e) = \bar{y}^{e,\theta}$ to $\mathcal{O}^K_{\psi}(G)$\;
			\If{$\bar{y}^{e,\theta} \notin Y_{\psi}$}{
				$Y_{\psi} \gets Y_{\psi} \cup \{\bar{y}^{e,\theta}\}$;
				\texttt{REC-OBS}$_\psi$($\bar{y}^{e,\theta}$)\;
			}
		}
	}
\end{algorithm}

Note that $ \mathcal{O}_{\psi}^K(G) $ is a nondeterministic automaton.
A deterministic automaton contained in $ \mathcal{O}_{\psi}^K(G) $ can be obtained by choosing one initial state and by making one sensing decision at each transition.
To formalize the reasoning above, we now establish that $ \mathcal{O}_{\psi}^K(G) $ collects all possible sensor activation policies guaranteeing $K$-delayed opacity.
\begin{definition}\label{contain}
	Given a sensor activation policy $ \Omega $ and a verification function $\psi$, the delay observer $ \mathcal{O}^K_{\Omega}(G) $ is said to be contained in $ \mathcal{O}^K_{\psi}(G) $ 
	if $ \forall s\in \matL(G) $, 
	\begin{equation}\label{ccontain}
	(I(\beta(P_\Omega(s))),\Theta(\omega(s))) \in \beta_{\psi}(P_\Omega(s)).
	\end{equation}
\end{definition} 
\vspace{1ex}

\begin{lemma}\label{LDO}
	The delay observer $ \mathcal{O}^K_{\Omega}(G) $ is contained in $ \mathcal{O}^K_{\psi}(G) $ if and only if \eqref{iffdc} holds.
\end{lemma}

\begin{proof}
	(\underline{Necessity}):
	Suppose that $ \mathcal{O}^K_{\Omega}(G) $ is contained in $ \mathcal{O}^K_{\psi}(G) $.
	Then, \eqref{ccontain} holds for any $ s\in \matL(G) $.
	By \eqref{vf} and \eqref{Opsi}, we have $ \psi(\beta(P_\Omega(s))) = O $, implying that \eqref{iffdc} holds.
	
	(\underline{Sufficiency}):
	Suppose that \eqref{iffdc} holds.
	Since $ \psi(y_0)=O $, there exists $ \bar{y}_0\in Y_{\psi,0} $ such that $ I(y_0)=I(\bar{y}_0) $ and $ \Theta(D(y_0)) = D_\psi(\bar{y}_0) $ (cf. lines 1-6 of Algorithm \ref{AOS} and lines 2-11 of Algorithm \ref{AO}).
	Suppose that \eqref{ccontain} holds for a string $ s=s'\in\matL(G) $.
	We now consider an $ e\in E $ such that $ \alpha(s'e)! $.
	In case $ e\notin \Theta(\omega(s')) $, we directly have that \eqref{ccontain} holds for $ s=s'e $ since $ P_\Omega(s')=P_\Omega(s'e) $ and $ \omega(s')=\omega(s'e) $.
	In case $ e\in \Theta(\omega(s')) $, let us denote $ \bar{y}=(I(\beta(P_\Omega(s'))),\Theta(\omega(s'))) $ and evaluate \texttt{REC-OBS}$ _\psi $($ \bar{y} $) in Algorithm \ref{AO}.
	First, lines 16-24 return $ (\emptyset,\Theta(\omega(s'e))) $.
	Then, lines 25-37 return $ \bar{y}^e=(I(\beta(P_\Omega(s'e))),\Theta(\omega(s'e))) $ based on $ \Theta(\omega(s'e)) $ (cf. lines 25-36 of Algorithm \ref{AOS}).
	By \eqref{vf}, since $ \psi(\beta(P_{\Omega}(s')))=O $, we have $ \psi(\bar{y}^e)=O $, implying that $ \bar{y}^e $ is added in $ Y_\psi $ (cf. lines 38-42).
Since $ \bar{y}^e\in \beta_\psi(\bar{y},e) $, $ \bar{y} $ will not be eliminated (cf. lines 13-14).
Hence, \eqref{ccontain} holds for $ s'e $.
	By induction, we conclude that $ \mathcal{O}^K_{\Omega}(G) $ is contained in $ \mathcal{O}^K_{\psi}(G) $.
\end{proof}

To find an optimal policy $ \bar{\Omega} $ based on $ \mathcal{O}^K_{\psi}(G) $, Algorithm \ref{APO} applies a greedy strategy  at each transition in $ \mathcal{O}^K_{\psi}(G) $ to determine the sensing decision.
The idea is to maximize the monitored events while guaranteeing $ K $-delayed opacity.
The following result shows that this greedy strategy returns a policy that solves Problem \ref{P}.

\begin{algorithm} [t]
	\caption{Sensor Activation Policy Optimization} \label{APO}
	\KwIn{$ \mathcal{O}_{\psi}^K(G) $;}
	\KwOut{$ \bar{\Omega}=(R,\Theta) $ with $ R=(Q,E,\omega,q_0) $;}
	
	Find $ \bar{y}_0 \in  Y_{\psi,0} $ satifying $ \forall \bar{y}_0' \in Y_{\psi,0}, D_{\psi}(\bar{y}_0') \not\subset D_{\psi}(\bar{y}_0) $\;
	$ i\gets0 $; $ \psi_Q(\bar{y}_0)\gets i $; $ \Theta(i)\gets D_\psi(\bar{y}_0) $; $ q_0 \gets i $\;
	\texttt{ADD}($ \bar{y}_0, \bar{\Omega}, \mathcal{O}_{\psi}^K(G) $)\;
	\For {$ q\in Q $}{
		\For {$ e\in E\backslash\{e\in E\mid \omega(q,e)! \} $}{
			Add $ \omega(q,e)=q $ to $ R $\;
		}
	}
	\textbf{Return} $ \bar{\Omega} $\;
	\vspace{0.6ex}
	\textbf{Procedure} \texttt{ADD}($\bar{y}, {\bar{\Omega}}, \mathcal{O}_{\psi}^K(G)$)\;
	\For{$ {e\in E_o: \beta_{\psi}(\bar{y},e)!} $}{
		Find $ \bar{y}'\in \beta_{\psi}(\bar{y},e) $ satisfying that $\forall \bar{y}''\in \beta_{\psi}(\bar{y},e), D_\psi(\bar{y}') \not\subset D_\psi(\bar{y}'') $\;
		\If{$ {\psi_Q(\bar{y}')\in\mathbb{N}} $}{
			Add $\omega(\psi_Q(\bar{y}),e) = \psi_Q(\bar{y}')$ to $R$\;
		}
		\Else{
			$ i \gets i+1 $; $ \psi_Q(\bar{y}')\gets i $; 
			$ \Theta(i)\gets D_\psi(\bar{y}') $\;
			$ Q \gets Q \cup \{i\} $; Add $ \omega(\psi_Q(\bar{y}),e)=i $ to $ R $\;
			\texttt{ADD}($ \bar{y}', {\bar{\Omega}}, \mathcal{O}_{\psi}^K(G) $)\;
		}
	}
\end{algorithm}

\begin{theorem}\label{TS}
	Algorithm \ref{AO} and Algorithm \ref{APO} solve Problem \ref{P}.
\end{theorem}

\begin{proof}
First, we show that Algorithm \ref{APO} will produce a non-empty output whenever Problem \ref{P} admits a solution.
Suppose that Problem \ref{P} has a solution $ {\bar{\Omega}} $.
Then, $ \mathcal{O}^K_{\bar{\Omega}}(G) $ is contained in $ \mathcal{O}^K_{\psi}(G) $ by Theorem \ref{TKDO} and Lemma \ref{LDO} since $ \mathcal{L}(G) $ is $ K $-delayed opaque w.r.t. $X_S$ and $ P_{\bar{\Omega}} $.
Hence, Algorithm \ref{AO} will produce a non-empty $ \mathcal{O}^K_{\psi}(G) $, which guarantees that Algorithm~\ref{APO} will produce a non-empty output, as $ {\bar{\Omega}} $ can be constructed directly from the states in $ \mathcal{O}^K_{\psi}(G) $ (cf. lines 1 and 10 of Algorithm \ref{APO}).

Let $ {\bar{\Omega}} $ be the output of Algorithm \ref{APO}.
Next, we show that such $ {\bar{\Omega}} $ indeed constitutes a solution to Problem~\ref{P}.
\\
	We first check condition i) in Problem \ref{P}.
	Since $ {\bar{\Omega}} $ is built directly with corresponding states in $ \mathcal{O}^K_{\psi}(G) $ (cf. lines 1-2 and lines 10-15 of Algorithm \ref{APO}), we have that $ \mathcal{O}^K_{\bar{\Omega}}(G) $ is contained in $ \mathcal{O}^K_{\psi}(G) $.
	Then, by Theorem \ref{TKDO} and Lemma \ref{LDO}, $ \mathcal{L}(G) $ is $ K $-delayed opaque w.r.t. $X_S$ and $ P_{\bar{\Omega}} $.
	\\
We then check condition ii) in Problem \ref{P}.
	By contradiction, suppose there exists $ \Omega'=(R',\Theta') $ with $ R'=(Q',E,\omega',q_0')  $ such that $ {\bar{\Omega}} \subset \Omega' $ and $ \mathcal{L}(G) $ is $ K $-delayed opaque w.r.t. $X_S$ and $ P_{\Omega'} $.
	Then, by Theorem \ref{TKDO} and Lemma \ref{LDO}, $ \mathcal{O}^K_{\Omega'}(G) = (Y', E_o, \beta', y_{0}') $ is contained in $ \mathcal{O}^K_{\psi}(G) $.
	Since $ {\bar{\Omega}} \subset \Omega' $, there exist two cases: $ \Theta(\omega(\epsilon)) \subset \Theta'(\omega'(\epsilon)) $ or $ \Theta(\omega(\epsilon)) = \Theta'(\omega'(\epsilon)) $.
	In the first case, $ \Theta(\omega(\epsilon)) $ cannot be selected in line 1, contradicting the fact that $ {\bar{\Omega}} $ is the output of Algorithm \ref{APO}.
	In the second case, there exists $ s\in\mathcal{L}(G)$ such that $ \Theta(\omega(se)) \subset \Theta'(\omega'(se)) $ and for any $ s' \in \bar{s},\ \Theta(\omega(s')) = \Theta'(\omega'(s')) $.
	Denote $ \bar{y}\!=\!(I(\beta(P_{\bar{\Omega}}(s))),\Theta(\omega(s)))\!=\!(I(\beta'(P_{\Omega'}(s))),\Theta'(\omega'(s))) $.
	Since $ \mathcal{O}^K_{\bar{\Omega}}(G) $ and $ \mathcal{O}^K_{\Omega'}(G) $ are contained in $ \mathcal{O}^K_{\psi}(G) $, we have $ \bar{y} \in \beta_{\psi}(P_\Omega(s)) $.
	Then, since $ \Theta(\omega(se)) \subset \Theta'(\omega'(se)) $, we have that $ (I(\beta(P_{\bar{\Omega}}(se))),\Theta(\omega(se))) \in \beta_{\psi}(\bar{y},e) $ cannot be selected in line 10, contradicting the fact that $ {\bar{\Omega}} $ is the output of Algorithm \ref{APO}.
	Hence, there exists no other policy $ \Omega $ such that $ {\bar{\Omega}}\subset\Omega $ and $ \mathcal{L}(G) $ is $ K $-delayed opaque w.r.t. $ X_S $ and $ P_{\Omega} $.
	We conclude that $ {\bar{\Omega}} $ is a solution to Problem \ref{P}.
\end{proof}

\begin{remark}
	(Non-uniqueness of solution).
	Note that Problem \ref{P} does not exclude the existence of more $ {\bar{\Omega}} $ satisfying i)-ii), meaning that the solution may be not unique. 
	This is reflected in the fact that $ \bar{y}_0 $ and $ \bar{y}^* $ in lines 1 and 10 of Algorithm \ref{APO} may be not unique.
\end{remark}

\begin{figure}[t]
	\centering
	\subfigure[$ {\mathcal{O}_{\psi}^{K_1}(G)} $]
	{\label{fmpo}
		\includegraphics[height=4.3cm]{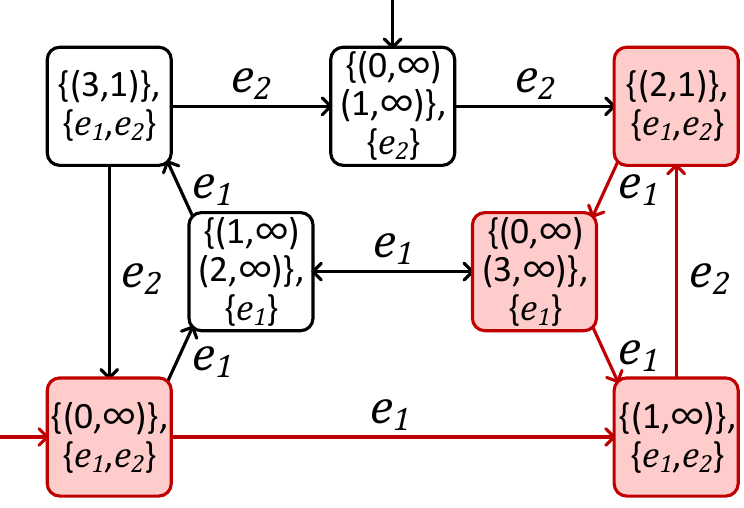} }
	\subfigure[$ {\bar{\Omega}} $]
	{\label{fop}
		\includegraphics[height=4.3cm]{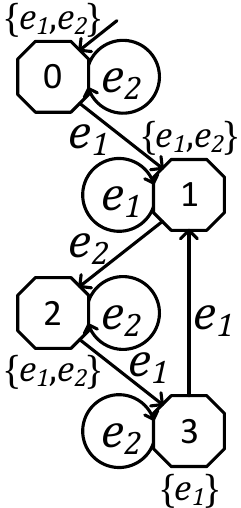}} 
	\caption{(a) Container $ \mathcal{O}_{\psi}^{K_1}(G) $ (with greedy strategy marked in red), collecting all possible sensor activation policies guaranteeing $K$-delayed opacity with $ K\!=\!K_1 $; (b) Optimal sensor activation policy $ {\bar{\Omega}} $ built with the marked container states.}
	\label{fdp}
\end{figure}

\begin{example}\label{EOP}
	The sensor activation policy in Fig. \ref{fsap} does not achieve $K$-delayed opacity with $K\!=\!K_1\!=\!1$ (cf. Example \ref{EKDO}). 
	To design a sensor activation policy achieving this, we obtain $ \mathcal{O}_{\psi}^{K_1}(G) $ in Fig. \ref{fmpo} from Algorithm \ref{APO}.
	Then, a greedy strategy is applied (marked in red) to obtain the sensor activation policy $ {\bar{\Omega}} $ in Fig. \ref{fop}.
	Note that in Fig. \ref{fmpo}, redundant sensing decisions (representing that some events being monitored cannot happen at the current step) are omitted.
	By Theorem \ref{APO}, $ {\bar{\Omega}} $ in Fig. \ref{fop} solves Problem \ref{P}.
\end{example}

\section{Conclusion}\label{8}
	This paper has investigated opacity of discrete event systems under delayed observations. 
	A novel notion of $ K $-delayed opacity has been introduced along with a delay observer structure to enable verification. 
	To synthesize and optimize the sensor activation policy, a new container structure has been constructed collecting all feasible sensor activation policies that preserve $ K $-delayed opacity. 
	The perspective adopted in this work is the one of a defender trying to mask the system to an eavesdropper. 
	An interesting point for future work is to adopt the perspective of the eavesdropper.

	\bibliographystyle{IEEEtran}
	\bibliography{References}

@book{cassandras2009introduction,
  title={Introduction to discrete event systems},
  author={Cassandras, Christos G and Lafortune, Stephane},
  year={2009},
  publisher={Springer Science \& Business Media}
}

@article{dallal2013most,
  title={On most permissive observers in dynamic sensor activation problems},
  author={Dallal, Eric and Lafortune, Stephane},
  journal={IEEE Transactions on Automatic Control},
  volume={59},
  number={4},
  pages={966--981},
  year={2013},
  publisher={IEEE}
}

@article{wang2010optimal,
  title={Optimal sensor activation for diagnosing discrete event systems},
  author={Wang, Weilin and Lafortune, St{\'e}phane and Girard, Anouck R and Lin, Feng},
  journal={Automatica},
  volume={46},
  number={7},
  pages={1165--1175},
  year={2010},
  publisher={Elsevier}
}

@article{sears2016minimal,
  title={Minimal sensor activation and minimal communication in discrete-event systems},
  author={Sears, David and Rudie, Karen},
  journal={Discrete Event Dynamic Systems},
  volume={26},
  number={2},
  pages={295--349},
  year={2016},
  publisher={Springer}
}

@article{yin2019general,
  title={A general approach for optimizing dynamic sensor activation for discrete event systems},
  author={Yin, Xiang and Lafortune, St{\'e}phane},
  journal={Automatica},
  volume={105},
  pages={376--383},
  year={2019},
  publisher={Elsevier}
}

@article{yin2019synthesis,
  title={Synthesis of dynamic masks for infinite-step opacity},
  author={Yin, Xiang and Li, Shaoyuan},
  journal={IEEE Transactions on Automatic Control},
  volume={65},
  number={4},
  pages={1429--1441},
  year={2019},
  publisher={IEEE}
}

@article{lafortune2018history,
  title={On the history of diagnosability and opacity in discrete event systems},
  author={Lafortune, St{\'e}phane and Lin, Feng and Hadjicostis, Christoforos N},
  journal={Annual Reviews in Control},
  volume={45},
  pages={257--266},
  year={2018},
  publisher={Elsevier}
}

@article{jacob2016overview,
  title={Overview of discrete event systems opacity: Models, validation, and quantification},
  author={Jacob, Romain and Lesage, Jean-Jacques and Faure, Jean-Marc},
  journal={Annual Reviews in Control},
  volume={41},
  pages={135--146},
  year={2016},
  publisher={Elsevier}
}

@article{wang2024distributed,
  author={Wang, Jiwei and Baldi, Simone and Yu, Wenwu and Yin, Xiang},
  journal={IEEE Transactions on Automatic Control}, 
  title={Distributed Fault Diagnosis in Discrete Event Systems With Transmission Delay Impairments}, 
  year={2024},
  volume={69},
  number={8},
  pages={5508-5515},
  doi={10.1109/TAC.2024.3369510}}

@inproceedings{wang2024fault,
  title={Fault Diagnosis and Prognosis in Partially-Observed Discrete Event Systems with Delayed Observations},
  author={Wang, Jiwei and Baldi, Simone and Yu, Wenwu and Yin, Xiang},
  booktitle={2024 IEEE 63rd Conference on Decision and Control (CDC)},
  pages={3675--3680},
  year={2024},
  organization={IEEE}
}

@inproceedings{saboori2007notions,
  title={Notions of security and opacity in discrete event systems},
  author={Saboori, Anooshiravan and Hadjicostis, Christoforos N},
  booktitle={2007 IEEE 46th Conference on Decision and Control},
  pages={5056--5061},
  year={2007},
  organization={IEEE}
}

@article{ji2018enforcement,
  title={Enforcement of opacity by public and private insertion functions},
  author={Ji, Yiding and Wu, Yi-Chin and Lafortune, St{\'e}phane},
  journal={Automatica},
  volume={93},
  pages={369--378},
  year={2018},
  publisher={Elsevier}
}

@ARTICLE{li2023opacity,
  author={Li, Xiaoyan and Hadjicostis, Christoforos N. and Li, Zhiwu},
  journal={IEEE Transactions on Automatic Control}, 
  title={Opacity Enforcement in Discrete Event Systems Using Extended Insertion Functions Under Inserted Language Constraints}, 
  year={2023},
  volume={68},
  number={11},
  pages={6797-6803},
  doi={10.1109/TAC.2023.3239433}}

@ARTICLE{xie2024optimal,
  author={Xie, Yifan and Li, Shaoyuan and Yin, Xiang},
  journal={IEEE Transactions on Automatic Control}, 
  title={Optimal Synthesis of Opacity-Enforcing Supervisors for Qualitative and Quantitative Specifications}, 
  year={2024},
  volume={69},
  number={8},
  pages={4958-4973},
  doi={10.1109/TAC.2023.3342880}}

@ARTICLE{li2025opacity,
  author={Li, Xiaoyan and Hadjicostis, Christoforos N. and Li, Zhiwu},
  journal={IEEE Transactions on Automation Science and Engineering}, 
  title={Opacity Enforcement in Discrete Event Systems Using Modification Functions}, 
  year={2025},
  volume={22},
  number={},
  pages={3252-3264},
  doi={10.1109/TASE.2024.3391617}}

@article{cassez2012synthesis,
  title={Synthesis of opaque systems with static and dynamic masks},
  author={Cassez, Franck and Dubreil, J{\'e}r{\'e}my and Marchand, Herv{\'e}},
  journal={Formal Methods in System Design},
  volume={40},
  pages={88--115},
  year={2012},
  publisher={Springer}
}

@article{mizoguchi2021abstraction,
  title={Abstraction-based control under quantized observation with approximate opacity using symbolic control barrier functions},
  author={Mizoguchi, Masashi and Ushio, Toshimitsu},
  journal={IEEE Control Systems Letters},
  volume={6},
  pages={2222--2227},
  year={2021},
  publisher={IEEE}
}

@article{duan2023edit,
  title={Edit mechanism synthesis for opacity enforcement under uncertain observations},
  author={Duan, Wei and Liu, Ruotian and Fanti, Maria Pia and Hadjicostis, Christoforos N and Li, Zhiwu},
  journal={IEEE Control Systems Letters},
  volume={7},
  pages={2041--2046},
  year={2023},
  publisher={IEEE}
}
	
\end{document}